\documentclass[11pt,reqno,oneside]{amsart}
\usepackage{amsmath}
\usepackage{amssymb}
\usepackage{amsthm}
\usepackage{framed}
\usepackage{color}
\usepackage{xifthen}
\usepackage{graphicx}
\usepackage{tikz}

\newtheorem{theorem}{Theorem}
\newtheorem{lemma}[theorem]{Lemma}

\newtheorem{proposition}[theorem]{Proposition}

\theoremstyle{definition}
\newtheorem{definition}[theorem]{Definition}
\newtheorem{property}[theorem]{Property}
\newtheorem{example}[theorem]{Example}

\newcommand{\real}{\mathbb{R}}
\newcommand{\sets}{\mathcal{S}}
\newcommand{\euc}{\mathbb{E}}
\newcommand{\ad}{\mathcal{A}}	
\newcommand{\id}{\mathrm{id}}
\newcommand{\tr}{\mathrm{tr}}

\def\verbose{0}
\def\comment#1{\ifthenelse{\verbose=1}{\textcolor{red}{\emph{#1}}}{}}
\def\sidenote#1{\ifthenelse{\verbose=1}{\marginpar{\raggedright{\textcolor{red}{#1}}}}{}}

\begin{document}

\title{A Hadwiger Theorem for Simplicial Maps}

\date{\today}

\author{P. Christopher Staecker}
\address{Department of Mathematics, Fairfield University, Fairfield CT, USA}
\email{cstaecker@fairfield.edu}

\author{Matthew L. Wright}
\address{Institute for Mathematics and its Applications, University of Minnesota, Minneapolis MN, USA}
\email{mlwright@ima.umn.edu}

\begin{abstract}
	We define the notion of \emph{valuation} on simplicial maps between geometric realizations of simplicial complexes in $\real^n$.
	Valuations on simplicial maps are analogous to valuations on sets.
    In particular, we define the Lefschetz volumes, which are analogous to the intrinsic volumes of subsets of $\real^n$.
	Our definition not only provides a generalization of the Lefschetz number, but also yields a Hadwiger-style classification theorem for all such valuations.
\end{abstract}



\maketitle

\section{Introduction}

The Euler characteristic of a polyhedron (a geometric realization of a simplicial complex) can be generalized to the \emph{Lefschetz number} of a selfmap of a polyhedron. 
The Lefschetz number $L(f)$ is a classical homological invariant which in some sense gives an oriented count of the number of fixed points of the selfmap; see the text by Brown \cite{Brown}. In particular, the Lefschetz number of the identity map equals the Euler characteristic.

A theme in topological fixed-point theory has been to consider properties of the Euler characteristic of a space and generalize to the Lefschetz number of selfmaps. A notable recent result is the paper of Arkowitz and Brown \cite{AB}, which gives axioms which uniquely define the Lefschetz number. Arkowitz and Brown's axioms directly generalize earlier axioms for the Euler characteristic by Watts in \cite{Watts}.

A simple consequence of the classic Hadwiger theorem is that the Euler characteristic is the unique valuation which has value 1 on any compact convex set. Recent work by the first author in \cite{Staecker} generalizes this fact to the Lefschetz number. The main result of \cite{Staecker} is that the Lefschetz number defined for selfmaps on polyhedra is unique with respect to the valuation (additivity) property, along with a specification of the value on simplices. That work considers only the Euler characteristic, which in Hadwiger's context is the dimension zero intrinsic volume.

The goal of this paper is to obtain an analogue of the full Hadwiger theorem in all dimensions in the setting of simplicial self-maps on polyhedra. The characterization of valuations in this context includes the Lefschetz number as the dimension zero case, together with a set of higher-dimensional valuations. These other valuations do not seem to appear in the fixed point theory literature, but nonetheless give meaningful information about the fixed point set. In particular they all satisfy a Lefschetz fixed point theorem.
Furthermore, our valuations on simplicial maps are analogous to the valuations on real-valued functions described in \cite{BGW}.

\section{Background}

Classically, \emph{valuations} have been studied on subsets of $n$-dimensional Euclidean space, $\real^n$.
A valuation is a way of assigning a notion of size to sets \cite{Schanuel}.
More formally, a valuation on a collection $\sets$ of subsets of $\real^n$ is an additive function $v: \sets \to \real$, meaning that
\begin{equation}\label{eq:additive}
	v(A) + v(B) = v(A \cap B) + v(A \cup B)
\end{equation}
whenever $A, B, A \cap B, A \cup B \in \sets$, such that $v(\emptyset)=0$.
Valuation $v$ is \emph{Euclidean-invariant} if $v(\phi A) = v(A)$ for all $A \in \sets$ and $\phi \in \euc_n$, the group of Euclidean motions (isometries) on $\real^n$.
A valuation is \emph{continuous} on $\sets$ if it is continuous with respect to convergence of sets in a specified metric (generally the Hausdorff metric).

The intrinsic volumes\footnote{Intrinsic volumes known by various names, including \emph{Hadwiger measures}, \emph{Minkowski functionals}, and \emph{quermassintegrale}. These concepts are equivalent up to normalization.} $\mu_0, \mu_1, \ldots, \mu_n$ are Euclidean-invariant valuations commonly defined on polyconvex subsets of $\real^n$ (a \emph{polyconvex} set is a finite union of compact convex sets).
The valuation $\mu_0$ is Euler characteristic, and $\mu_n$ is Euclidean $n$-dimensional volume.
Intuitively, $\mu_k(A)$ gives a notion of the $k$-dimensional size of a set $A$, so $\mu_1$ measures length or perimeter, $\mu_2$ measures area, etc.
Formally, the intrinsic volumes are often defined via the \emph{Steiner formula}.
For any convex set $A \subset\real^n$, let $A_\rho$ denote the volume of the ``tube'' around $A$ of radius $\rho > 0$; that is,
\begin{equation*}
	A_\rho = \{ x \in \real^n \mid d(x,A) \le \rho \}.
\end{equation*}
The Steiner formula states that the volume of $A_\rho$ is a polynomial in $\rho$, whose coefficients involve the intrinsic volumes:
\begin{equation}\label{eq:Steiner}
	\mathrm{Vol}(A_\rho) = \sum_{k=0}^n \rho^{n-k} \omega_{n-k} \mu_k(A),
\end{equation}
where $\omega_{j}$ is the volume of the $j$-dimensional unit ball.
For a proof of the Steiner formula, see Schneider \cite{Schneider}.

The well-known Hadwiger theorem states that the space of Euclidean-invariant valuations on polyconvex sets in $\real^n$, continuous on convex sets with respect to the Hausdorff metric, is a vector space spanned by the intrinsic volumes \cite{Hadwiger, KR}.
\begin{theorem}[Hadwiger]
	Any Euclidean-invariant convex-continuous valuation $v$ on compact polyconvex subsets of $\real^n$ is a linear combination of the intrinsic volumes:
	\begin{equation}\label{eq:Hadwiger}
		v = \sum_{k=0}^n a_k \mu_k,
	\end{equation}
	for some constants $a_k \in \real$.
\end{theorem}

Our work will be in the setting of abstract finite simplicial complexes and their geometric realizations. Let $S$ be a finite set, and $P(S)$ the power set. A \emph{simplicial complex} (or simply \emph{complex}) is a subset $X\subset P(S)$ such that if $x \in X$ and $y\subseteq x$ then $y \in X$. A complex containing exactly one maximal element is a simplex. A one-element subset of $P(S)$ is called an \emph{open simplex}.

For any $a\in P(S)$, let $\bar a \subset P(S)$ be the simplex whose maximal element is $a$. For a simplex $\bar a$, the \emph{boundary} $\partial \bar a$ is the complex $\bar a - \{a\}$. The one-element set $\{a\}$ is the \emph{interior} of $\bar a$, denoted $\mathring a$. Let a \emph{polyhedron} refer to a geometric realization of a simplicial complex. 
When speaking of valuations, we sometimes use the terms \emph{simplicial complex} and \emph{polyhedron} interchangeably; in this paper, we always pair a simplicial complex with a geometric realization, and a polyhedron always comes with a simplicial decomposition. 
Since the original set $S$ is assumed to be finite, our polyhedra will always be compact.

Thus polyhedra are polyconvex sets, and intrinsic volumes of polyhedra are well-defined and well-studied in the literature.
Less commonly known is that any valuation on polyhedra induces a valuation on their interiors. 
If $v$ is any valuation on polyhedra, then we can extend $v$ to geometric realizations of open simplices as follows.
Let $\bar{x}$ denote a (closed, geometric) simplex.
Let $\mathring{x}$ denote the interior of $\bar{x}$ and $\partial \bar{x}$ denote the boundary of $\bar{x}$.
A closed simplex is equal to the \emph{disjoint} union of its interior and its boundary; that is, $\bar{x} = \mathring{x} \sqcup \partial \bar{x}$.
Furthermore, the boundary $\partial \bar{x}$ is itself a simplicial complex.
Then we can define
\begin{equation}\label{eq:val_open}
	v(\mathring{x}) = v(\bar{x}) - v(\partial \bar{x}).
\end{equation}
This allows us to write the valuation of a simplicial complex as a sum over all open simplices.

\begin{lemma}\label{lem:open_decomp}
	Let $v$ be a valuation on polyhedra, and let $X$ be a polyhedron.
    Then
    \begin{equation}\label{eq:open_decomp}
    	v(X) = \sum_{a \subseteq X} v(\mathring{a}),
    \end{equation}
    where the sum is over all simplices, of all dimensions, in $X$. 
\end{lemma}
(Throughout the paper, summations indexed by subsets of a complex, as in \eqref{eq:open_decomp}, will be assumed to range over only subsimplices of the complex.)

\begin{proof}
Let $w(X) = \sum_{a \subseteq X} v(\mathring a) = \sum_{a \subseteq X} v(a) - v(\partial a)$. Since $v$ is a valuation, we see that $w$ is also a valuation. We show that $w(X) = v(X)$ by induction on the number of simplices of $X$.

For the base case, when $X$ is empty we have $w(X) = v(X) = 0$.  For the inductive case, assume that $X$ has $k+1$ simplices, and let $a \subset X$ be a maximal simplex and $Y$ a complex of $k$ simplices with $X = a \cup Y$. Then by additivity of $v$ and $w$ and the induction hypothesis we have
\[ v(X) = v(a \cup Y) = v(a) + v(Y) - v(a\cap Y) = w(a) + w(Y) - w(a\cap Y) = w(X). \]
\end{proof}

Specifically, the intrinsic volumes of geometric realizations of open simplices are related to the intrinsic volumes of their closures as follows.
\begin{proposition}\label{prop:int_vol_open}
	Let $x$ be a geometric realization of a closed $n$-simplex and $\mathring{x}$ be its interior.
    Then
    \begin{equation}\label{eq:int_vol_open}
    	\mu_k(\mathring{x}) = (-1)^{n-k} \mu_k(x),
    \end{equation}
    for any $k \in \{ 0, 1, \ldots, n\}$.
\end{proposition}
\begin{proof}
We will prove the $k=0$ case, which is easily accessible with combinatorial arguments. A full proof for all $k$ appears in \cite{BGW} (see Remark 5 in that paper) using an integral-geometric approach to the intrinsic volumes. 

In the case for $k=0$, the intrinsic volume $\mu_0$ is the Euler characteristic $\chi$. When $x$ is a simplex we have $\chi(x) = 1$ and so the statement to be proved is that $\chi(\mathring x) = (-1)^n$. But $\chi(\mathring x) = \chi(x) - \chi(\partial x) = 1 - \chi(\partial x)$. Since $\partial x$ is topologically a sphere of dimension $n-1$, its Euler characteristic is $1+(-1)^{n-1}$. Thus we have $\chi(\mathring x) = 1 - (1 + (-1)^{n-1}) = (-1)^n$ as desired.
\end{proof}

The Euler characteristic $\mu_0$ resulting from Proposition \ref{prop:int_vol_open} is the \emph{combinatorial} Euler characteristic, which takes value $(-1)^n$ on an open $n$-simplex.
This Euler characteristic is unfortunately not a homotopy-type invariant, but it is a homeomorphism invariant and appears in the literature \cite{BG:pnas, BG:siam, VDD}.

\section{Valuations on Simplicial Maps}

Let $\ad$ be the set of \emph{admissible pairs} $(f, X)$, where $X$ is a polyhedron in $\real^n$ and $f: X \to X$ is a simplicial map (taking simplices to simplices).

We define valuations on admissible pairs as follows:

\begin{definition}
	A \emph{valuation} on admissible pairs is an additive function $v: \ad \to \real$ such that $v(f, \emptyset) = 0$. By \emph{additive}, we mean that $v(f,X) + v(f,Y) = v(f,X \cap Y) + v(f,X \cup Y)$ for any polyhedra $X$ and $Y$ in $\real^n$ such that $f$ is a simplicial map on each of $X$ and $Y$.
\end{definition}

Just as any valuation on polyhedra induces a valuation on open simplices (equation \eqref{eq:val_open}), any valuation on admissible pairs extends in a natural way to open simplices. Suppose that $x \subset X$ is a simplex and $f: x \to X$ is a simplicial map.
Then we define
\begin{equation}\label{eq:open_closed}
	v(f,\mathring{x}) = v(f,\bar{x}) - v(f, \partial \bar{x}),
\end{equation}
where $\mathring{x}$ refers to the interior of $\bar{x}$ as before.


We obtain a useful decomposition of any valuation on admissible pairs as a sum of values on open simplices.
For any admissible pair $(f,X)$ and valuation $v$, 
\begin{equation}\label{eq:open_sum}
    	v(f,X) = \sum_{x \subseteq X} v(f,\mathring{x}) = \sum_{x \subseteq X} \left( v(f,\bar{x}) - v(f, \partial \bar{x})\right)
    \end{equation}
The proof of equation \eqref{eq:open_sum} is essentially the same as that of Lemma \ref{lem:open_decomp}.

    
    
    

We highlight the following properties, which a valuation $v$ on admissible pairs may satisfy:

\begin{property}[Simplex]
	Valuation $v$ satisfies the \emph{simplex} property if
    \begin{equation}\label{eq:simplex}
    	v(f,x) = c(f,x)v(\id, \mathring{x}) + v(f, \partial x)
    \end{equation}
   	for any (closed) simplex $x \in X$ with interior $\mathring{x}$, where 
	\begin{equation}\label{eq:norm}
        c(f,x) = 
        \begin{cases}0 & \text{if } x \ne f(x) \\ 1 & 
         	\text{if } x = f(x) \text{ preserving orientation} \\ -1 &
           	\text{if } x = f(x) \text{ reversing orientation.}
        \end{cases}
    \end{equation}
Compare this with a similar property of the Lefschetz number in \cite{Staecker}.
By equation \eqref{eq:open_closed}, this property is equivalent to the statement $v(f,\mathring{x}) = c(f,x) v(\id,\mathring{x})$.
\end{property}

\begin{property}[Invariance]
	Valuation $v$ satisfies the \emph{invariance} property if 
    $v(f,X) = v(\phi \circ f \circ \phi^{-1}, \phi(X))$ 
    for any isometry $\phi$ of $\real^n$ and any polyhedron $X$.
    This means that if $v$ is an invariant valuation, then $v(f,X)$ does not depend on the particular location or orientation of $X$ in $\real^n$.
\end{property}

\begin{property}[Continuity]
	Valuation $v$ is \emph{continuous} if $v(\id, X_i)$ converges to $v(\id, X)$ for any sequence of convex polyhedra $(X_i)_i$ that converges in the Hausdorff metric to a convex polyhedron $X$.
\end{property}
    


\section{Lefschetz Volumes}

We highlight a class of valuations on admissible pairs.
For an admissible pair $(f, X)$ with $X \subset \real^n$, and for any $k \in \{0, 1, \ldots, n \}$ we define the \emph{$k$-dimensional Lefschetz volume} $v_k$:
\begin{equation}\label{eq:lefschetzvolume}
	v_k(f, X) = \sum_{x \subseteq X} (-1)^{\dim(x)-k} c(f,x) \mu_k(x),
\end{equation}
where the sum is over all simplices $x \subseteq X$, and $c(f,x)$ is given in \eqref{eq:norm}. In order to make sense of $\mu_k(x)$ for $k > 0$, we recall that $\real^n$ is equipped with the Euclidean metric.

Viewed in terms of the chain maps $f_q$, the number $c(f,x)$ is the coefficient on $x$ in the chain $f_q(x)$.
Thus the trace $\tr(f_q)$ equals the sum $\sum_{x\in C_q(X)}c(f,x)$. Let $V_{q,k}$ be the diagonal matrix consisting of the intrinsic volumes $\mu_k(x)$ for each $q$-simplex $x$. Then we have the following trace formula for the Lefschetz volumes:
\begin{equation}\label{eq:trace}
	v_k(f,X) = \sum_{q=0}^{\dim X} (-1)^{q - k}\tr(f_q V_{q,k}).
\end{equation}

The intrinsic volumes are higher-dimensional analogues of the Euler characteristic which are geometric rather than purely topological. In the same way, the Lefschetz volumes are higher-dimensional analogues of the Lefschetz number, which is itself a generalization of the Euler characteristic. 

Observe that setting the ``dimension'' $k$ to zero in \eqref{eq:trace} makes $V_{q,0}$ equal to the identity matrix, since the dimension zero intrinsic volume is the Euler characteristic, which is 1 for any simplex. Thus we have 
\[ v_0(f,X) = \sum_{q=0}^{\dim X} (-1)^q \tr f_q, \]
which is the classical Lefschetz number of $f$.

Observe also that if we use the identity map $\id$ for $f$ in \eqref{eq:lefschetzvolume} we have $c(\id,x) = 1$ for all $x$, and thus
\begin{equation}\label{eq:vk_id}
	v_k(\id, X) = \sum_{x \subseteq X} (-1)^{\dim (x) - k} \mu_k(x) = \sum_{x \subseteq X} \mu_k(\mathring x) = \mu_k(X).
\end{equation}

Thus, the Lefschetz volumes simultaneously generalize the Lefschetz number and the intrinsic volumes. The real number $v_k(f,X)$ may be thought of as a ``signed'' dimension-$k$ intrinsic volume of those simplices fixed by $f$.

The Lefschetz volumes maintain the most notable characteristic of the classical Lefschetz number, the following fixed point property:
\begin{theorem}
For any $k$, if $v_k(f, X) \ne 0$, then $f$ has a fixed point in a simplex of $X$ with dimension at least $k$.
\end{theorem}

\begin{proof}
If $v_k(f, X) \ne 0$, then there is some simplex $x \subset X$ with $c(f,x) \ne 0$ and $\mu_k(x) \neq 0$. Since $c(f,x)$ is nonzero, $f$ maps $x$ to itself. By the Brouwer fixed point theorem, $f$ has a fixed point in $x$. Since $\mu_k(x)$, the dimension $k$ intrinsic volume, is nonzero, $x$ must be a simplex of at least dimension $k$.
\end{proof}

Furthermore, the Lefschetz volumes satisfy each of the properties defined in the previous section.

\begin{proposition}
	The Lefschetz volumes are invariant, continuous valuations on admissible pairs that satisfy the simplex property.
\end{proposition}

\begin{proof}
	Additivity of $v_k$ follows from the fact that any polyhedron can be decomposed into a disjoint union of open simplices; thus $v_k$ is a valuation on admissible pairs.
    
    We now show that $v_k$ satisfies the simplex property.
    For any simplex $x$,
    \begin{align*}
    	v_k(f,x) &= \sum_{a \subseteq x} (-1)^{\dim(a)-k} c(f,a) \mu_k(a) \\
        	&= (-1)^{\dim(x)-k} c(f,x) \mu_k(x) + \sum_{a \subseteq \partial x} (-1)^{\dim(a)-k} c(f,a) \mu_k(a) \\
            &= c(f,x) \mu_k(\mathring{x}) + v_k(f,\partial x).
    \end{align*}
    To show that the above equals $c(f,x) v_k(\id,\mathring{x}) + v_k(f,\partial x)$, it suffices to show that $\mu_k(\mathring x) = v_k(\id,\mathring{x})$. 
    Applying equation \eqref{eq:vk_id} to both $x$ and $\partial x$, we obtain
    \begin{equation*}
    	\mu_k(\mathring{x}) 
        = \mu_k(x) - \mu_k(\partial x)
        = v_k(\id, x) - v_k(\id, \partial x) 
        = v_k(\id, \mathring{x}).
    \end{equation*}
    Thus, $v_k$ satisfies the simplex property.
    The continuity property of $v_k$ follows directly from \eqref{eq:vk_id} and continuity of $\mu_k$.
   
   It remains to prove invariance of $v_k$, that is $v_k(\phi \circ f \circ \phi^{-1}, \phi(X)) = v_k(f,X)$ where $\phi$ is any isometry. By \eqref{eq:trace}, we have:
   \[ v_k(\phi\circ f \circ \phi^{-1}, \phi(X)) = \sum_{q=0}^{\dim \phi(X)} (-1)^{q-k} \tr((\phi\circ f \circ \phi^{-1})_{q}W_{q,k}), \]
   where $W_{q,k}$ is the diagonal matrix of intrinsic volumes of simplices of $\phi(X)$. Since $\phi$ is an isometry these are the same as the intrinisic volumes of simplices of $X$, so $W_{q,k} = V_{q,k}$, where $V_{q,k}$ is the diagonal matrix of intrinsic volumes of simplices of $X$. Also since $\phi$ is an isometry we have $\dim(X) = \dim \phi(X)$. By naturality of the chain maps we have $(\phi\circ f\circ \phi^{-1})_q = \phi_q f_q \phi^{-1}_q$, and since $V_{q,k}$ is a diagonal matrix and the trace is similarity invariant we have
\begin{align*} 
v_k(\phi\circ f \circ \phi^{-1}, \phi(X)) &= \sum_{q=0}^{\dim X} (-1)^{q-k} \tr(\phi_q f_q \phi^{-1}_{q}V_{q,k}) 
=  \sum_{q=0}^{\dim X} (-1)^{q-k} \tr(\phi_q f_q V_{q,k} \phi^{-1}_{q} ) \\
&=  \sum_{q=0}^{\dim X} (-1)^{q-k} \tr(f_q V_{q,k}) = v_k(f,X).
\end{align*}
   
\end{proof}

\section{A Hadwiger Theorem}
We now turn to a Hadwiger-style classification theorem for valuations on admissible pairs.

\begin{theorem}\label{thm:HadwigerSimplicial}
Any invariant, continuous valuation $v$ on admissible pairs satisfying the simplex property is a linear combination of the Lefschetz volumes. That is, $v$
 can be written as
	\begin{equation}
		v(f,X) = \sum_{k=0}^{\dim X} a_k v_k(f,X)
	\end{equation}
	for some real constants $a_k$.
\end{theorem}

Note that the constants $a_k$ in the theorem depend only on $v$, not on $x$.

\begin{proof}

	We write $v(f,X)$ as a sum of values on open simplices of $X$ (by equation \eqref{eq:open_sum}) and apply the simplex property:
	\begin{equation}\label{eq:proofsum1}
		v(f,X) = \sum_{ x \subseteq X} v(f,\mathring x) = \sum_{ x \subseteq X} c(f,x) v(\id,\mathring x).
	\end{equation}
	
	Let $w(X) = v(\id, X)$, which is a real-valued function on polyhedra in $\real^n$.
    Furthermore, we will show that $w(x)$ is additive, Euclidean-invariant, and continuous.

	Additivity of $w$ follows from additivity of $v$.
	For any polyhedra $X$ and $Y$,
	\begin{align*}
		w(X) + w(Y) &= v(\id, X) + v(\id, Y) \\ 
		&= v(\id, X \cap Y) + v(\id, X \cup Y) = w(X \cap Y) + w(X \cup Y).
	\end{align*}
	
	Invariance of $w$ follows from invariance of $v$.
	For any isometry $\phi \in \euc_n$ and any simplex $x \in \real^n$,
	\begin{equation*}
		w(\phi (x)) = v(\id, \phi (x)) = v(\phi \circ \id \circ \phi^{-1}, \phi (x)) = v(\id, x) = w(x).
	\end{equation*}
	
    Continuity of $w$ follows from continuity of $v$.
    Continuity of $v$ implies that $w$ is a Hausdorff-continuous valuation on compact convex polyhedra.
    Since compact convex polyhedra are Hausdorff-dense among compact convex sets in $\real^n$, Hausdorff-continuity allows us to extend $w$ to a valuation on compact convex sets in $\real^n$.
	
    Therefore, $w$ is a Euclidean-invariant valuation on compact polyconvex sets, continuous on convex sets with respect to the Hausdorff metric.
	The classic Hadwiger Theorem then implies that for any  polyhedron $X$,
	\begin{equation}\label{eq:proofsum2}
		w(X) = \sum_{k=0}^{\dim X} a_k \mu_k(X)
	\end{equation}
	for some constants $a_k \in \real$ (that depend only on $w$, not on $X$).
    In particular, equation \eqref{eq:proofsum2} holds when $X$ is a closed simplex $x$ or its boundary $\partial x$.
    We obtain:
    \begin{align}
		v(\id, \mathring x) 
        &= w(\mathring x) 
        = w(x) - w(\partial x) \notag \\
        &= \sum_{k=0}^{\dim x} a_k \mu_k(x) - \sum_{k=0}^{\dim x} a_k \mu_k(\partial x)
        = \sum_{k=0}^{\dim x} a_k \mu_k(\mathring x). \label{eq:proofsum3}
	\end{align}
	
Thus by \eqref{eq:proofsum1} we have
\begin{align*}
v(f,X) = \sum_{ x \subseteq X} c(f,x) \sum_{k=0}^{\dim x} a_k \mu_k(\mathring x) = \sum_{k=0}^{\dim X} a_k \sum_{x \subseteq X} c(f,x) \mu_k(\mathring x)
\end{align*}
where the exchanging of the sums is valid because the $a_k$ depend only on $v$, not on $x$.
Now by Proposition \ref{prop:int_vol_open} the above becomes
\[ v(f,X) = \sum_{k=0}^{\dim X} a_k \sum_{x \subseteq X} c(f,x) (-1)^{(\dim x) - k}\mu_k(x)  = 
\sum_{k=0}^{\dim X} a_k v_k(x)
\]
as desired.
\end{proof}

\section{Examples}



We now give a few explicit computations of the Lefschetz volumes for simple admissible pairs.

\begin{example}\label{ex:nontrivial}
    Consider the following space $X$, illustrated in Figure \ref{fig:nontrivial}, consisting of vertices $p_1,\dots, p_5$, edges $e_1, \dots, e_6$, and a single face $s$ (shaded), where all edges have length 1 and the face has area $\sqrt{3}/4$ (according to the Euclidean metric).
    \begin{figure}[hbt]
        \begin{center}
            \begin{tikzpicture}
                \def\s{2.5}; 
                \path[draw] (150:\s) coordinate [label=left:$p_1$] (v1)
                    -- (0:0) coordinate [label=above:$p_2$] (v2)
                    -- (210:\s) coordinate [label=left:$p_5$] (v5)
                    -- cycle;
                
                \draw[fill=lightgray] (v2)
                    -- (30:\s) coordinate [label=right:$p_3$] (v3)
                    -- (-30:\s) coordinate [label=right:$p_4$] (v4)
                    -- cycle;
                
                \foreach \point in {v1, v2, v3, v4, v5}
                    \fill [black] (\point) circle (2pt);
                    
                \node[above] at (150:\s*0.5) {$e_1$};
                \node[above] at (30:\s*0.5) {$e_2$};
                \node[right] at (\s*0.87,0) {$e_3$};
                \node[below] at (-30:\s*0.5) {$e_4$};
                \node[below] at (210:\s*0.5) {$e_5$};
                \node[left] at (\s*-0.87,0) {$e_6$};
                \node at (0.5*\s,0) {$s$};
            \end{tikzpicture}
        \end{center}
        \caption{Polyhedron $X$ discussed in Example \ref{ex:nontrivial}.}
        \label{fig:nontrivial}
    \end{figure}
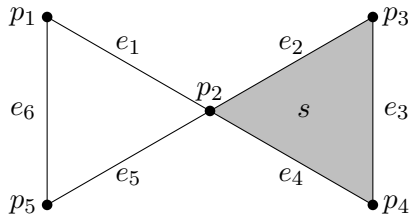
    
    
    Let $f:X\to X$ be the map which is the identity on edges $e_1, e_5$, and $e_6$, and does a vertical flip of the face. The induced maps on chain groups for $f$ are as follows:
    \begin{align*}
    f_0 &= \begin{bmatrix} 
    1 & 0 & 0 & 0 & 0 \\
    0 & 1 & 0 & 0 & 0 \\
    0 & 0 & 0 & 1 & 0 \\
    0 & 0 & 1 & 0 & 0 \\
    0 & 0 & 0 & 0 & 1
    \end{bmatrix} \\
    f_1 &= \begin{bmatrix}
    1 & 0 &  0 & 0 & 0 & 0 \\
    0 & 0 &  0 & 1 & 0 & 0 \\
    0 & 0 & -1 & 0 & 0 & 0 \\
    0 & 1 &  0 & 0 & 0 & 0 \\
    0 & 0 &  0 & 0 & 1 & 0 \\
    0 & 0 &  0 & 0 & 0 & 1
    \end{bmatrix} \\
    f_2 &= \begin{bmatrix} -1 \end{bmatrix}
    \end{align*}
    
    The traces of these matrices are 3, 2, and $-1$. 
    Thus, by formula \eqref{eq:trace}, the Lefschetz number is
    \begin{equation*}
    v_0(f,X) = L(f,X) = 3 - 2 + (-1) = 0.
    \end{equation*}
    
    \sidenote{maybe this part can be improved}
    To compute $v_1(f,X)$, we first remind ourselves of the value $\mu_1$ for each simplex.
    If $p$ is a vertex, $\mu_1(p)=0$ since points have no length.
    For an edge $e$ of unit length, $\mu_1(e)=1$.
    For the face $s$, $\mu_1(s)=3/2$, which is half the perimeter of $s$.
    Thus, 
    \begin{equation*}
    	v_1(f,X) = \sum_{q=0}^2 (-1)^{q-1} \tr(f_q V_{q,1}) = 0 + 2 + \frac{3}{2} = \frac{7}{2}.
    \end{equation*}
    
    Lastly, to compute $v_2$, we consider $\mu_2$.
    Since vertices $p$ and edges $e$ have no area, $\mu_2(p) = \mu_2(e) = 0$.
    The area of the face $s$ is $\mu_2(s) = \sqrt{3}/4$, so we obtain
    \begin{equation*}
    	v_2(f,X) = \sum_{q=0}^2 (-1)^{q-2} \tr(f_q V_{q,2}) = 0 + 0 + (-1)\frac{\sqrt{3}}{4} = -\frac{\sqrt{3}}{4}.
    \end{equation*}
%
\end{example}

\begin{example}\label{ex:zero}
	Let $Y$ be the space, illustrated in Figure \ref{fig:zero}, consisting of vertices $p_1, \ldots, p_6$ and edges as shown.
    The edges connecting vertices $p_1, \ldots, p_4$ are each of unit length, while the edge connecting $p_5$ and $p_6$ has length $5$ (the lengths of the remaining edges are irrelevant for this example).
	\begin{figure}[hbt]
        \begin{center}
            \begin{tikzpicture}
                \def\s{0.7}; 
                
                \path[draw] (0,0) coordinate [label=right:$p_4$] (v4)
                    -- (120:\s) coordinate [label=above:$p_2$] (v2)
                    -- (180:\s) coordinate [label=left:$p_1$] (v1)
                    -- (240:\s) coordinate [label=below:$p_3$] (v3)
                    -- (v4) -- (v1);
                
               \path[draw] (v4)
                    -- (1.5*\s,2.5*\s) coordinate [label=right:$p_5$] (v5)
                    -- (1.5*\s,-2.5*\s) coordinate [label=right:$p_6$] (v6)
                    -- cycle; 
                
                \foreach \point in {v1,v2,v3,v4,v5,v6}
                    \fill [black] (\point) circle (2pt);
                    
            \end{tikzpicture}
        \end{center}
        \caption{Polyhedron $Y$ discussed in Example \ref{ex:zero}.}
        \label{fig:zero}
    \end{figure}
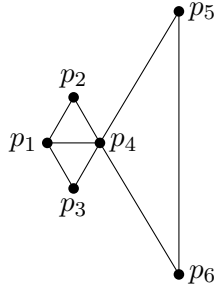
	
    Let $g: Y \to Y$ be the simplicial map that is the identity on $p_1, \ldots p_4$, and interchanges $p_5$ and $p_6$.
    This $g$ fixes four vertices and five edges, and maps one edge to itself with opposite orientation.
    The Lefschetz volume $v_0(g,Y)$ is then
    \begin{equation*}
    	v_0(g,Y) = 4 + 5(-1) -1(-1) = 0.
    \end{equation*}
    
    To compute $v_1(g,Y)$, we must take into account the lengths of the edges.
    Since $g$ fixes the five edges of length $1$ and reverses the edge of length $5$, we find
	\begin{equation*}
    	v_1(g,Y) = 5(1) -1(5) = 0.
    \end{equation*}
    We see that \emph{all} Lefschetz volumes of $(g,Y)$ are zero, and yet $g$ has a fixed point.
    We note that the lengths of the edges are important in this example.
    It is possible to construct a simplicial complex with all edges of unit length and no nonzero Lefschetz volumes.
\end{example}

While the Lefschetz number is independent of the simplicial decomposition of a polyhedron, the other (``higher-dimensional'') Lefschetz volumes are \emph{not} independent of such decomposition, as the following example shows.

\begin{example}\label{ex:decomp}
	Let $X$ be a polyhedron consisting of an edge of length $2$, and let $f: X \to X$ the simplicial map that flips $X$, interchanging its endpoints.
    Let $Y$ be a polyhedron that results from subdividing $X$ by adding a vertex in the center, and let $g: Y \to Y$ be the simplicial map that flips $Y$.
    (See Figure \ref{fig:decomp}.)
    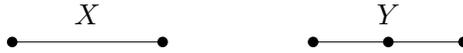
\begin{figure}[hbt]
        \begin{center}
            \begin{tikzpicture}
                \def\s{1}; 
                
                \path[draw] (0,0) coordinate (v1)
                    -- (2*\s,0) coordinate (v2);
                    
                \node[above] at (\s,0.1) {$X$};
                    
				\path[draw] (4*\s,0) coordinate (v3)
                    -- (5*\s,0) coordinate (v4)
                    -- (6*\s,0) coordinate (v5);
                
                \node[above] at (5*\s,0.1) {$Y$};
                
                \foreach \point in {v1,v2,v3,v4,v5}
                    \fill [black] (\point) circle (2pt);
                    
            \end{tikzpicture}
        \end{center}
        \caption{Polyhedra discussed in Example \ref{ex:decomp}.}
        \label{fig:decomp}
    \end{figure}
    
    Observe that $v_1(f,X)=-2$ because one edge of length $2$ is mapped to itself with opposite orientation.
    However, $v_1(g,Y)=0$ because no edges are mapped to themselves.
    Thus, $v_1$ is not invariant with respect to simplicial decomposition.
    This contrasts with the situation for Lefschetz numbers; indeed, $v_0(f,X) = v_0(g,Y) = 1$.
\end{example}



\section{Further Work}

While the Lefschetz number is invariant with respect to both homotopy and simplicial decomposition, the other Lefschetz volumes do not possess such invariance. It would be interesting to determine if there is some other, analogous, conception of invariance satisfied by the Lefschetz volumes.

It is desirable to extend the valuation theory described in this paper to continuous self-maps of topological spaces, as is possible for the Lefschetz number.
However, since the Lefschetz volumes depend on the simplicial decomposition of polyhedra (see Example \ref{ex:decomp}), it is unclear how to approach continuous maps via simplicial approximations.

It also remains to study valuations that do not satisfy the simplex property, or that satisfy other properties. The authors suspect that there are more settings, not yet considered, in which Hadwiger-type theorems can be proved.
\sidenote{say more here?}

\section{Acknowledgements}

The authors thank Vin de Silva for his helful comments, including a simplified proof of Lemma \ref{lem:open_decomp}. The second author gratefully acknowledges the support of the Institute for Mathematics and its Applications (IMA), where the author was a postdoctoral fellow during the IMA's annual program on applications of algebraic topology.


\bibliographystyle{amsalpha}

\end{document}